\documentclass[12pt,reqno]{amsart}

\usepackage{amssymb}
\usepackage{tikz}
\usetikzlibrary{arrows}

\numberwithin{equation}{section}

\def\ZZ{{\mathbb Z}}

\def\FF{{\mathbb F}}

\def\Hom{\operatorname{Hom}}

\def\pfi{{\varphi}}

\newtheorem{lemma}{Lemma}[section]
\newtheorem{corollary}[lemma]{Corollary}
\newtheorem{theorem}[lemma]{Theorem}
\newtheorem{proposition}[lemma]{Proposition}

\theoremstyle{definition}

\title{Duality and syzygies for semimodules over numerical semigroups}

\author{Julio Jos\'e Moyano-Fern\'andez}
\address{Institut f\"ur Mathematik, Universit\"at Osnabr\"uck. Albrechtstrasse 28a,
D-49076 Osnabr\"uck, Germany}
\email{jmoyano@uos.de}

\author{Jan Uliczka}
\address{Institut f\"ur Mathematik, Universit\"at Osnabr\"uck. Albrechtstrasse 28a,
D-49076 Osnabr\"uck, Germany}
\email{juliczka@uos.de}

\begin{document}

\subjclass[2010]{Primary 20M14; Secondary 05A19}
\keywords{Numerical semigroup, $\Gamma$-semimodule, dual semimodule, syzygy}
\thanks{The first author was partially supported by the Spanish Government Ministerio de Educaci\'on y Ciencia (MEC), grants MTM2007-64704 and MTM2012--36917--C03--03 in cooperation with the European Union in the framework of the founds ``FEDER''}

\maketitle

\begin{abstract}
Let $\Gamma=\langle \alpha, \beta \rangle$ be a numerical semigroup. 
In this article we consider the dual $\Delta^*$ of a $\Gamma$-semimodule $\Delta$; in particular we deduce a formula that expresses the minimal set of generators of $\Delta^*$ in terms of the generators of $\Delta$.  As applications we compute the minimal graded free resolution of a graded $\FF[t^{\alpha},t^{\beta}]$-submodule of $\FF[t]$, and we investigate the structure of the selfdual $\Gamma$-semimodules, leading to a new way of counting them.
\end{abstract}

\section{Introduction} \label{section1}

In our paper \cite{mu2} we investigated relations between $\langle \alpha ,\beta \rangle$-semimodules and lattice paths from $(0,\alpha)$ to $(\beta, 0)$. In particular we introduced the syzygy $\mathrm{Syz}(\Delta)$ of a semimodule $\Delta$. This concept seems to be related to the notion of syzygy in commutative algebra: An $\langle \alpha ,\beta \rangle$-semimodule $\Delta$ may be identified with a finitely generated graded submodule $M_{\Delta}$ of the polynomial ring $\FF[t]$ (over some field $\FF$) considered as a module over the subalgebra $\FF[t^{\alpha},t^{\beta}]$. The first syzygy of $M_{\Delta}$ is nonzero exactly in those degrees which are elements of $\mathrm{Syz}(\Delta)$. 
Further investigation of this relationship was the initial motivation for the current paper.
Another construction for semimodules, the so-called dual, turned out to be a useful tool for this task.
\medskip

The paper is organised as follows: In Section \ref{section2} some basic facts on homomorphisms of $\langle \alpha ,\beta \rangle$-semimodules are collected. The next section deals with the dual of a $\langle \alpha ,\beta \rangle$-semimodule; the central result (Theorem \ref{3T1}) is the description of a set  of generators for the dual $\Delta^{*}$ in terms of the minimal set of generators for $\Delta$. We continue with two applications of this result. In Section \ref{section4} we consider the above mentioned problem of commutative algebra and compute the minimal graded free resolution of an $\FF[t^{\alpha},t^{\beta}]$-module $M_{\Delta}$ associated to the semimodule $\Delta$. The second application is a counting of the selfdual semimodules (i.~e.~semimodules with $\Delta^* \cong \Delta$); here we recover a result of Gorsky and Mazin \cite{gm}. The paper closes with some observations concerning the interplay of the semimodule operations \emph{dual} and \emph{syzygy}.
\medskip

For the convenience of the reader we begin with a brief review of prerequisites and relevant material from our previous papers \cite{mu,mu2}. 
A \emph{numerical semigroup} $\Gamma$ is a submonoid of the additive monoid $\mathbb{N}$ such that $L:=\mathbb{N} \setminus \Gamma$ is a finite set. The elements of $L$ are called \emph{gaps} of $\Gamma$.  A \emph{$\Gamma$-semimodule} $\Delta$ is a non-empty subset of $\mathbb{Z}$ that is bounded below and satisfies $\Delta + \Gamma \subseteq \Delta$. A \emph{system of generators} of $\Delta$ is a subset $\mathcal{E}$ of $\Delta$ with
\[
\bigcup_{x \in \mathcal{E}} (x+\Gamma) = \Delta. 
\]
It is called \emph{minimal} if no proper subset of $\mathcal{E}$ generates $\Delta$. Every $\Gamma$-semimodule has a unique minimal system of generators $\mathcal{B}$; since $\Delta \setminus \Gamma$ is finite, so is $\mathcal{B}$. If $\min \Delta =0$ then the nonzero elements of $\mathcal{B}$ are gaps of $\Gamma$ and the positive difference of two such gaps is a gap as well. Subsets of $\mathbb{N}$ of this type are called \emph{$\Gamma$-lean}.
\medskip

We are mainly interested in numerical semigroups with two generators $\alpha < \beta$. In this case a lemma of Rosales \cite[Lemma 1]{rosales} yields a very useful characterisation for the numbers not contained in $\langle \alpha, \beta \rangle$: for $\ell \in \mathbb{Z}$ we have $\ell \notin \langle \alpha, \beta \rangle$ if and only if there exist $a, b \in \mathbb{N}_{>0}$ such that $\ell=\alpha \beta -a \alpha -b \beta$. 
Therefore the positive difference of two gaps $i_k=\alpha \beta -a_k \alpha -b_k \beta$, $k=1,2$ is a gap if and only if $a_1-a_2$ and $b_1-b_2$ bear different signs. Hence we can define a partial ordering in the set of gaps:
\[
i_1 <_L i_2 \ \ \ : \Longleftrightarrow \ \ \ a_1 > a_2 \wedge b_1 < b_2
\]
for $i_1,i_2 \in L$. In particular we may assume that the gaps in a $\Gamma$-lean set are ordered increasingly with respect to $<_L$.

\section{Homomorphisms of $\Gamma$-semimodules} \label{section2}

Let $\Delta$ and $\Delta'$ be $\Gamma$-semimodules.
A map $f:~\Delta \to \Delta'$ is called a \emph{homomorphism} of $\Gamma$-semimodules (or \emph{$\Gamma$-homomorphism}) if
\[ f(\gamma + x) = \gamma + f(x) \]
holds for any $\gamma \in \Gamma$ and $x \in \Delta$. It is easily seen that all homomorphisms of $\Gamma$-semimodules share the same structure:

\begin{lemma} A map  $f:~\Delta \to \Delta'$ between $\Gamma$-semimodules  is a $\Gamma$-homomorphism if and only if there exists $c \in \ZZ$ such that
$f(x) = x+c$ for all $x \in \Delta$. In particular, any homomorphism of $\Gamma$-semimodules is injective.
\end{lemma}

\begin{proof} If $f(x) =  x+c$ holds for all $x \in \Delta$, then $f$  is a $\Gamma$-homomorphism, since  for any $\gamma \in \Gamma$ we have
\[ f(\gamma + x) = (\gamma +x) +c = \gamma + (x+c) = \gamma + f(x). \]
Conversely, let $f$ be a $\Gamma$-homomorphism. We show that the difference $f(x) - x$ does not depend on $x \in \Delta$: Let $x_1,x_2 \in \Delta$. Since $\Gamma$ and $\Delta$ have only finitely many gaps, there exist $\gamma_1, \gamma_2 \in \Gamma$ such that 
\[ \gamma_1 + x_1 = \gamma_2 + x_2.\] This implies
\[ \gamma_1 + f(x_1) = f(\gamma_1 + x_1) = f(\gamma_2 + x_2) = \gamma_2 + f(x_2), \]
and subtraction of these equations yields $f(x_1) - x_1 = f(x_2) - x_2$. 
\end{proof}

Since any homomorphism of $\Gamma$-semimodules is injective, such a map is an isomorphism of semimodules if and only if it is surjective.
\medskip

By the previous lemma, we may identify the set of $\Gamma$-homomorphisms
with a subset of $\ZZ$ by mapping a $\Gamma$-homomorphism $f$ to its $c$ with $f(x) = x+c$, namely
\[ \Hom_{\Gamma}(\Delta,\Delta') \cong \left\{ c \in \ZZ ~ \mid ~ c +  \Delta \subseteq \Delta' \right\}. \]

\begin{proposition}
The set  $\Hom_{\Gamma}(\Delta,\Delta')$ is a  $\Gamma$-semimodule.
\end{proposition}

\begin{proof} Let $c \in \Hom_{\Gamma}(\Delta,\Delta')$, then $x+c \in \Delta'$ for any $x \in \Delta$. Since $\Delta'$ is a $\Gamma$-semimodule, we also  have 
$\gamma + (x+c) = x + (\gamma + c) \in \Delta'$ for arbitrary $\gamma \in \Gamma$, but this implies $\gamma + c \in  \Hom_{\Gamma}(\Delta,\Delta')$.
\end{proof}

Shifting of $\Delta$ or $\Delta'$ also shifts $ \Hom_{\Gamma}(\Delta,\Delta')$, with the usual contra- and covariant behaviour of Hom:

\begin{proposition} \label{prop2.3}
For any $d, d' \in \ZZ$ we have
\[
\Hom_{\Gamma}(\Delta + d,\Delta'+d') = \Hom_{\Gamma}(\Delta,\Delta') -d +d'.
\]
\end{proposition}

\begin{proof}
This follows from
\begin{eqnarray*}
& & c \in \Hom_{\Gamma}(\Delta + d,\Delta'+d') \\
& \iff & \forall x \in \Delta:~ (x+d) +c-d' \in \Delta' \\
& \iff & \forall x \in \Delta:~ x+(c +d-d') \in \Delta' \\
& \iff& c+d-d' \in \Hom_{\Gamma}(\Delta ,\Delta') \\
& \iff& c \in \Hom_{\Gamma}(\Delta ,\Delta') -(d-d').
\end{eqnarray*}
\end{proof}

\section{Duality} \label{section3}

The most important $\Gamma$-homomorphisms are those into the semimodule $\Gamma$. For any $\Gamma$-semimodule $\Delta$ we set
\[ \Delta^* := \Hom_{\Gamma}(\Delta,\Gamma) \cong \left\{ c \in \ZZ ~ \mid ~ c + \Delta  \subseteq \Gamma\right\}, \]
which we call the \emph{dual} of $\Delta$.
\begin{proposition} \label{3P1}
Let  $\Delta, \Delta_1, \Delta_2 $ be  $\Gamma$-semimodules, and let $d \in \ZZ$. Then
\begin{itemize}
\item[(a)] $\left(\Delta + d \right)^* = \Delta^* - d$.
\item[(b)] $\left(\Delta_1 \cup \Delta_2 \right)^* = \Delta_1^* \cap \Delta_2^*$.
\item[(c)] $\Gamma^* = \Gamma$.
\end{itemize}
\end{proposition}

\begin{proof}
Part (a) follows immediately from Proposition \ref{prop2.3}, and (b) holds since
\begin{eqnarray*}
& & c \in \left(\Delta_1 \cup \Delta_2 \right)^* \\
& \iff & \forall x \in  \Delta_1 \cup \Delta_2:~ x+c \in \Gamma \\
& \iff & ( \forall x \in  \Delta_1 :~ x+c \in \Gamma) \wedge ( \forall x \in   \Delta_2:~ x+c \in \Gamma) \\
& \iff & x \in \Delta_1^* \wedge x \in \Delta_2^* \\
& \iff & x \in \Delta_1^* \cap  \Delta_2^*.
\end{eqnarray*}
For (c)  we note that the inclusion $\Gamma \subseteq \Gamma^*$ is clear since $\Gamma$ is closed under addition. Conversely let $x \in \Gamma^*$,
then $x + \gamma \in \Gamma$ for all $\gamma \in \Gamma$, so in particular $x = x+0 \in \Gamma$, which implies the reverse inclusion.
\end{proof}

\begin{corollary}
Let $I$ be a $\Gamma$-lean set and $\Delta_I = \bigcup_{i \in I} \left( \Gamma + i \right)$, then $\Delta_I^* =  \bigcap_{i \in I} \left( \Gamma - i \right)$.
\end{corollary}

From now on we only consider numerical semigroups $\Gamma=\langle \alpha , \beta \rangle$. The main result of this section describes the minimal system of generators of a dual $\Delta^*$ in terms of  the minimal system of generators of $\Delta$.

\begin{theorem} \label{3T1}
Let $I=\{0,i_1, \ldots ,i_n\}$ be a $\Gamma$-lean set with gaps $i_k = \alpha \beta - a_k \alpha - b_k \beta$ which are ordered increasingly with respect to $<_L$,
and let $\Delta_I =   \bigcup_{i \in I} \left( \Gamma + i \right)$, then
\begin{equation}\label{eq:3T1}
\Delta_I^* = \left(\Gamma + a_1 \alpha  \right) \cup \bigcup_{k=1}^{n-1} \left(\Gamma + a_{k+1} \alpha +b_k \beta \right) \cup \left(\Gamma + b_n \beta \right). 
\end{equation}
\end{theorem}

\begin{proof}

By the previous proposition we have $\Delta_I^* = \Gamma \cap \bigcap_{k=1}^n \left(\Gamma - i_k \right)$, hence an $x = r\alpha + s \beta \in \Gamma$ is contained in $\Delta_I^*$ if and only if
\[y_k := x+ i_k = \alpha \beta - (a_k-r) \alpha - (b_k -s) \beta \in \Gamma \]
holds for $k = 1, \ldots , n$. By Rosales' characterisation of gaps (see \cite{rosales}) this means
\begin{equation} \label{2eq1}
 a_k -r \leq 0 ~ \vee ~ b_k - s \leq 0 ~~\mbox{for}~~ k=1, \ldots ,n.
\end{equation}
First we show $\Delta_I^*$ is contained in the $\Gamma$-semimodule on the right-hand side of (\ref{eq:3T1}), which we denote by  $\widetilde{\Delta_I}$: Let $x = r \alpha + s \beta \in \Delta_I^*$, then (\ref{2eq1}) holds. Since, by the ordering of the gaps $i \in I$, $r \geq a_k$ implies $r > a_{\ell}$ for all $\ell > k$ and $s \geq b_k$ implies $s > b_{\ell}$ for all $\ell < k$, we distinguish three cases:
\begin{itemize}
\item[i)] If $r \geq a_1$, then $x = a_1 \alpha + (r-a_1) \alpha + s \beta \in \Gamma + a_1 \alpha$.
\item[ii)] If  $a_k > r \geq a_{k+1}$  for some $k \in \{1, \ldots , n \}$ then condition (\ref{2eq1}) implies $s \geq b_k$, and hence 
$x = a_{k+1} \alpha + b_k \beta + \big( (r-a_{k+1}) \alpha + (s-b_k) \beta \big ) \in \Gamma + a_{k+1} \alpha + b_k \beta$.
\item[iii)] Otherwise $a_n >r$, hence condition (\ref{2eq1}) implies $s \geq b_n$,  and so $x = b_n \beta +r \alpha + (s-b_n) \beta \in \Gamma + b_n \beta$.
\end{itemize}

Conversely we have to show that all generators of $\widetilde{\Delta_I}$ are contained in $\Gamma - i_{\ell}$ for $\ell = 1, \ldots ,n$. This is easily seen for
$a_1 \alpha$ and $b_n \beta$ since
\begin{eqnarray*}
a_1 \alpha + i_{\ell} &=& \alpha \beta - \underbrace{  (a_{\ell}-a_1)}_{ \leq 0}\alpha - b_{\ell} \beta \in \Gamma \\
b_n \beta + i_{\ell} &=& \alpha \beta -a_{\ell} \alpha -\underbrace{  (b_{\ell}-a_n)}_{ \leq 0}  \beta \in \Gamma,
\end{eqnarray*}
and for the other generators 
\[ a_{k+1} \alpha + b_k \beta + i_{\ell} = \alpha \beta - (a_{\ell} - a_{k+1}) \alpha - (b_{\ell} - b_k) \beta \]
is also an element of $\Gamma$, because  $b_{\ell} - b_k \leq 0$ for $k \geq \ell$ and $a_{\ell} - a_{k+1} \leq 0$ for $k< \ell$. 
\end{proof}

As a first application of Theorem \ref{3T1} we show that the dual of a dual $\Gamma$-semimodule is the original semimodule:

\begin{proposition} \label{3P2}
Let $I=\{0,i_1, \ldots ,i_n\}$ be as in Theorem \ref{3T1}. Then
\[
(\Delta_I^*)^*=\Delta_I.
\]
\end{proposition}

\begin{proof}
Let $\Delta_I=\Gamma \cup \bigcup_{k=1}^{n} (\Gamma + i_k)$. It follows from Theorem \ref{3T1} that
\[
\Delta_I^* =(\Gamma + a_1\alpha) \cup \bigcup_{k=1}^{n-1} \left(\Gamma + a_{k+1} \alpha +b_k \beta \right) \cup \left(\Gamma + b_n \beta \right);
\]
in other words
\begin{eqnarray*}
& &\Delta_I^*-a_1\alpha \\
&=&\Gamma \cup \bigcup_{k=1}^{n-1} \left(\Gamma + (a_{k+1}-a_1) \alpha +b_k \beta \right) \cup \left(\Gamma + b_n \beta -a_1\alpha \right)\\
&=& \Gamma \cup \bigcup_{k=1}^{n-1} \left(\Gamma + \alpha \beta -(a_1-a_{k+1}) \alpha -(\alpha-b_k) \beta  \right) \cup \left(\Gamma + \alpha \beta -a_1\alpha-(\alpha -b_n) \beta ) \right).
\end{eqnarray*}
In order to have the generators of $\Delta_I^*-a_1\alpha$ ordered increasingly with respect to $<_L$ we set
\begin{equation}\label{eq:ik}
\hat{\imath}_k := \alpha \beta - (a_1-a_{n-k+2}) \alpha - (\alpha -b_{n-k+1}) \beta, 
\end{equation}
therefore
\[
\Delta_I^*-a_1\alpha=\Gamma \cup \bigcup_{k=1}^{n} (\Gamma + \hat{\imath}_k).
\]
By Proposition \ref{3P1}(a) one has $(\Delta_I^{*})^{*}+a_1\alpha=(\Delta_I^{*}-a_1\alpha)^*$; on the other hand Theorem \ref{3T1} implies
\begin{eqnarray*}
& & \Gamma \cup \bigcup_{k=1}^{n} (\Gamma + \hat{\imath}_k) \\
&=& (\Gamma + a_1\alpha) \cup \bigcup_{k=1}^{n-1} \left(\Gamma + (a_1-a_{n-k+1}) \alpha +(\alpha-b_{n-k+1}) \beta \right) \cup \left(\Gamma + (\alpha - b_1) \beta \right) \\
&=& a_1\alpha + \left (\Gamma \cup \bigcup_{k=2}^{n-1}(\Gamma+ \alpha \beta -a_{n-k+1}\alpha - b_{n-k+1}\beta) \cup (\Gamma+ \alpha \beta -a_1\alpha-b_1\beta)\right )\\
&=& a_1\alpha+\Delta_I,
\end{eqnarray*}
and therefore $(\Delta_I^*)^*=\Delta_I$, as desired.
\end{proof}

\section{Syzygies of $\FF[\Gamma]$-modules and of $\Gamma$-semimodules} \label{section4}

The numerical semigroup $\Gamma = \langle \alpha, \beta \rangle$ and the $\Gamma$-semimodules correspond to objects of commutative algebra:
Let $\FF$ be an arbitrary field, then we can consider the semigroup ring $\FF[\Gamma]$, which may be identified with the subalgebra $R:=  \FF[t^{\alpha},t^{\beta}]$ of the polynomial ring; the counterparts of the $\Gamma$-semimodules in this setting are the graded $R$-submodules of $\FF[t]$. In this section we investigate minimal graded free resolutions of such modules and their syzygies.
\medskip

Let $I$ be a $\Gamma$-lean set as in the previous theorem, and let $M_I = \sum_{i \in I} Rt^i$. We consider the first syzygy of $M_I$, the kernel of the map
\begin{eqnarray*}
\bigoplus_{i \in I} R(-i) & \stackrel{\pfi_1}{\longrightarrow} & M_I \\
(f_0, \ldots , f_n) & \longmapsto & \sum_{k =0}^n f_k t^{i_k}.
\end{eqnarray*}
By a result of Piontkowski (see \cite{pio}) this kernel is generated by \emph{bivectors}, i.e.~elements of the form $(0, \ldots, 0, t^{\gamma_k},0, \ldots, 0, t^{\gamma_m},0, \ldots,0)$ with $i_k + \gamma_k = i_m + \gamma_m$. By Corollary 4.3 of our article \cite{mu2}, in fact $n+1$ special bivectors are sufficient, namely
\begin{eqnarray*}
f_0 &=& (t^{(\beta-a_1)\alpha},-t^{b_1 \beta},0, \ldots ,0) \\
f_k &=& (0, \ldots, 0, t^{(a_k-a_{k+1})\alpha},-t^{(b_{k+1}-b_k) \beta},0, \ldots ,0)~~~\mbox{for}~k=1, \ldots, n \\
f_n &=& (-t^{(\alpha -b_n) \beta},0, \ldots, 0,t^{a_n \alpha}).
\end{eqnarray*}
 
The degrees $\deg f_k = j_k$ are exactly the elements of the set $J$, the second component of the \emph{fundamental couple} $[I,J]$, see \cite{mu}. Hence, as already mentioned in the introduction, the support of the syzygy $\ker \pfi_1$ agrees with the $\Gamma$-semimodule $\Delta_J$, the object we called the syzygy of $\Delta_I$ in our article \cite{mu2}.

The second step of the free resolution of $M_I$ is the map
\begin{eqnarray*}
\bigoplus_{j \in J} R(-j) & \stackrel{\pfi_2}{\longrightarrow} & \ker \pfi_1 \\
(g_0, \ldots , g_n) & \longmapsto & \sum_{k=0}^n g_kf_k.
\end{eqnarray*}
The condition $\pfi_2(g_0, \ldots, g_n) = 0$ yields the following system of equations:
\[
\begin{array}{lclcl}
g_0 t^{(\beta -a_1) \alpha} & -& g_n t^{(\alpha - b_n) \beta} &=& 0 \\
g_1 t^{(a_1 -a_2) \alpha}& -& g_0 t^{b_1 \beta} &=& 0 \\
g_k t^{(a_k - a_{k+1}) \alpha} & -& g_{k-1} t^{(b_k - b_{k-1}) \beta} &=& 0 \, \, \,\, \, \, \mbox{for}~ k=2, \ldots , n-1 \\
g_n t^{a_n \alpha}& -& g_{n-1} t^{(b_n-b_{n-1}) \beta} &=& 0
 \end{array} \]
We can solve for $g_0$ and get
\begin{eqnarray*}
g_k &=& g_0 t^{b_k \beta - (a_1-a_{k+1}) \alpha}~~~\mbox{for}~ k=1, \ldots,  n-1 \\
g_n &=& g_0  t^{b_n \beta - a_1 \alpha},
\end{eqnarray*}
as one easily checks by induction on $k$. Hence $g=(g_0, \ldots , g_n)$ is an element of $\ker \pfi_2$ if and only if it can be written in the form
\[ g= g_0 \left(1, t^{b_1 \beta - (a_1-a_{2}) \alpha}, \ldots, t^{b_{n-1} \beta - (a_1-a_n) \alpha},  t^{b_n \beta - a_1 \alpha} \right) \]
with some $g_0 \in R$ such that all the entries are in $R$ as well.

In the language of $\Gamma$-semimodules this means that we are looking for the dual of the semimodule
\begin{equation} \label{4eq1}
\widehat{\Delta_I} := \Gamma \cup \bigcup_{k=1}^{n-1} \left( \Gamma + (b_k \beta -(a_1-a_{k+1}) \alpha \right) \cup (\Gamma +  b_n \beta - a_1 \alpha).
\end{equation}
Again this dual can be computed using Theorem \ref{3T1}: As in the proof of Proposition \ref{3P2} we renumber the generators to get them into ascending order with respect to $<_L$, i.~e.~we set
\begin{eqnarray*}
\hat{\imath}_1 &:=& \alpha \beta - a_1 \alpha - (\alpha -b_n) \beta \\
\hat{\imath}_k &:=& \alpha \beta - (a_1-a_{n-k+2}) \alpha - (\alpha -b_{n-k+1}) \beta \ \  \mbox{~for~} \  k=2, \ldots , n.
\end{eqnarray*}
Theorem \ref{3T1} then implies
\begin{eqnarray*}
\widehat{\Delta_I}^* &=& \left( \Gamma + \bigcup_{k=1}^n (\Gamma + \hat{\imath}_k ) \right)^* \\
&=& a_1 \alpha + \Gamma \cup  \left( \bigcup_{k=1}^{n-1} \Gamma + (a_1 - a_{n-k+1}) \alpha + (\alpha - b_{n-k+1}) \beta \right) \cup
(\Gamma + (\alpha -b_1) \beta) \\
&=& a_1 \alpha + \Gamma \cup  \left( \bigcup_{k=1}^{n-1} \Gamma + (a_1 - a_{k+1}) \alpha + (\alpha - b_{k+1}) \beta \right) \cup
(\Gamma + (\alpha -b_1) \beta) \\
&=& a_1 \alpha + \left( \Gamma \cup  \left( \bigcup_{k=1}^{n-1} \Gamma + \alpha \beta -   a_{k+1}  \alpha -  b_{k+1}  \beta \right) \cup
(\Gamma + \alpha \beta - a_1\alpha -b_1 \beta) \right)\\
&=& a_1 \alpha + \left(\Gamma \cup \bigcup_{k=2}^n (\Gamma + i_k) \cup (\Gamma + i_1) \right) \\
&=& a_1 \alpha + \Delta_I,
\end{eqnarray*}
hence
\begin{eqnarray*}
\ker \pfi_2 &=& \left\{ g_0 \left(1, t^{b_1 \beta - (a_1-a_{2}) \alpha}, \ldots, t^{b_{n-1} \beta - (a_1-a_n) \alpha},  t^{b_n \beta - a_1 \alpha} \right)
~ \mid ~ g_0 \in M_I \cdot t^{a_1 \alpha} \right\}, \\
& \cong& M_I.
\end{eqnarray*}

Therefore we have shown the following
\begin{theorem}\label{4T1}
Let $\Gamma=\langle \alpha, \beta \rangle$ be a numerical semigroup. 
Let $I$ be a $\Gamma$-lean set, and let $M_I = \sum_{i \in I} Rt^i$ with $R=  \FF[t^{\alpha},t^{\beta}]$. Then
the minimal graded free resolution of $M_I$ is ---up to a shift--- periodic of period $2$.
\end{theorem}

Of course this result could alternatively be deduced from the graded version of 
a classical result by Eisenbud, see \cite{eis}, Theorem 6.1 (ii). In particular, Theorem \ref{4T1} implies that the relationship between the first syzygy in the commutative algebra sense and the first syzygy in the setting of $\Gamma$-semimodules does not extend to the higher syzygies, since the sequence of semimodule syzygies $\mathrm{Syz}^{(k)}(\Delta_I)$, $k \in \mathbb{N}$, is periodic with period $n+1=|I|$.

\section{Selfdual $\Gamma$-semimodules} \label{section5}

In this section we want to investigate the so-called selfdual $\Gamma$-semimodules, that is, $\Gamma$-semimodules $\Delta$ such that $\Delta^* \cong \Delta$. Our main tool will be the description of the $\Gamma$-semimodule by a two-row matrix as introduced in our paper \cite{mu2}. We briefly recall the basic facts here.
\medskip

Let $\Delta$ be a $\Gamma$-semimodule minimally generated by a $\Gamma$-lean set $I$. The gaps $i_k$ can be mapped to lattice points via $i_k=\alpha\beta -a_k \alpha -b_k \beta \mapsto (a_k,b_k)$ and these images can be viewed as turns from $x$-direction to $y$-direction in a lattice path from $(0,\alpha)$ to $(\beta,0)$.
This yields a bijection between the set of isomorphism classes of $\Gamma$-semimodules and the set of lattice paths from $(0,\alpha)$ to $(\beta, 0)$ staying below the diagonal, cf.~Theorem 3.7 in \cite{mu2}.
\medskip

Any such a lattice path can also be described by a matrix with two rows, where the $i$-th column contains the numbers of steps downwards and to the right the path takes between the $(i-1)$-th and the $i$-th turning point; the entries in the rows of this matrix sum up to $\beta$~resp.~$\alpha$.
\\
\begin{center}
  \begin{tikzpicture}[scale=0.6]
    \draw[] (0,0) grid [step=1cm](7,5);
    \draw[] (0,5) -- (7,0);
    \draw[ultra thick] (0,5) -- (0,3) -- (1,3) -- (1,2) -- (3,2) -- (3,1) -- (4,1) -- (4,0) -- (7,0);
    \draw[fill=white] (0,5) circle [radius=0.1]; 
    \draw[fill] (1,3) circle [radius=0.1]; 
     \draw[fill] (3,2) circle [radius=0.1]; 
     \draw[fill] (4,1) circle [radius=0.1]; 
      \draw[fill=white] (7,0) circle [radius=0.1]; 
      \node at (12,2.5) {$\left (
\begin{array}{cccc}
2&1&1 &1\\
1&2&1 &3
\end{array}
\right )$};
   \end{tikzpicture}
\end{center}
\begin{center}
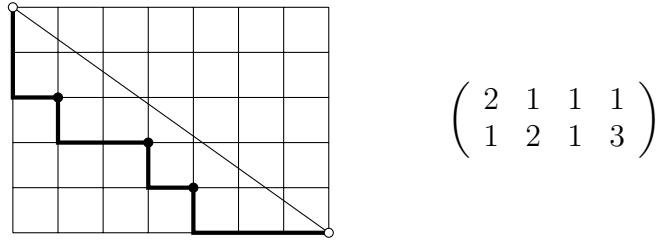
{\small Lattice path and matrix for the $\langle 5,7 \rangle$-semimodule generated by $\{0,9,6,8\}$.}
\end{center}
\medskip

In general the matrix associated to a $\Gamma$-lean set $I=\{i_0=0, i_1, \ldots , i_n\}$ with $i_k=\alpha\beta -a_k \alpha -b_k \beta$ is 

$$
\left (
\begin{array}{ccccc}
\alpha -b_n&b_n-b_{n-1}&\ldots &b_2-b_1&b_1\\
a_n&a_{n-1}-a_n&\ldots&a_1-a_2&\beta-a_1
\end{array}
\right ).
$$

Conversely for any two-row matrix with positive integer entries and the row sum property there exists exactly one cyclic permutation of its columns such that the permuted matrix describes a lattice path staying below the diagonal. Therefore we even have a bijection between the set of isomorphism classes of  $\Gamma$-semimodules and the set of equivalence classes of two-row matrices modulo cyclic permutation of columns. In the sequel this equivalence will be denoted by $\equiv$.
\medskip

Next we deduce how taking the dual changes the associated matrix. Let $\Delta$ be a semimodule generated by a $\Gamma$-lean set $I$ as above. Then a semimodule isomorphic to $\Delta^*$ is generated by $\hat{I}=\{0,\hat{\imath}_1, \ldots , \hat{\imath}_n\}$ where $\hat{\imath}_1=\alpha \beta -a_1\alpha-(\alpha-b_n)\beta$ and $\hat{\imath}_k=\alpha \beta -(a_1-a_{n-k+2})\alpha-(\alpha-b_{n-k+1})\beta$ for $k=2,\ldots ,n$, as already mentioned in the proof of Proposition \ref{3P2}. Hence the matrix associated to $\hat{I}$ is equivalent to
\begin{eqnarray*}
& & 
\left (
\begin{array}{ccccc}
\alpha -(\alpha-b_1)&(\alpha-b_1)-(\alpha-b_{2})&\ldots &(\alpha-b_{n-1})-(\alpha-b_n)&\alpha-b_n\\
a_1-a_2&(a_{1}-a_3)-(a_1-a_2)&\ldots&a_1-(a_{1}-a_n)&\beta-a_1
\end{array}
\right )\\
&=& 
\left (
\begin{array}{ccccc}
b_1&b_2-b_{1}&\ldots &b_n-b_{n-1}&\alpha -b_n\\
a_1-a_2&a_{2}-a_3&\ldots&a_n&\beta-a_1
\end{array}
\right ).
\end{eqnarray*}

Comparison of the matrices shows that the dual has the order of the columns reversed with an additional shift by one position in the lower row; the effect of taking the dual can therefore be described by
\begin{equation}\label{eq:matrix}
\left (
\begin{array}{ccccc}
y_0&y_1&\ldots &y_{n-1}&y_n\\
x_0&x_1&\ldots &x_{n-1}&x_n
\end{array}
\right )^* \equiv
\left (
\begin{array}{ccccc}
y_n&y_{n-1}&\ldots &y_1&y_0\\
x_{n-1}&x_{n-2}&\ldots &x_{0}&x_n
\end{array}
\right ).
\end{equation}

Hence a $\Gamma$-semimodule is selfdual if and only if its matrix remains unchanged (up to cyclic permutation) under the inversion and shifting described in (\ref{eq:matrix}). This means that there exists $k \in \mathbb{N}$ such that
\begin{equation*}
\left (
\begin{array}{ccccc}
y_k&y_{k+1}&\ldots &y_{k-2}&y_{k-1}\\
x_k&x_{k+1}&\ldots &x_{k-2}&x_{k-1}
\end{array}
\right ) =
\left (
\begin{array}{ccccc}
y_n&y_{n-1}&\ldots &y_1&y_0\\
x_{n-1}&x_{n-2}&\ldots &x_{0}&x_n
\end{array}
\right ),
\end{equation*}
with the matrices being equal entrywise and \emph{not} only up to cyclic permutation. Comparison of the entries shows $y_{k+i}=y_{n-1}$ for $i=0,\ldots, n-k$, $y_i=y_{k-1-i}$ for $i=0,\ldots, k-1$. Therefore the top row of the matrix consists of two palindromic blocks. For further investigation we distinguish three cases: 
\medskip

\noindent Case I: If $n$ is even the matrix consists of an odd number of columns, hence one of the palindromic blocks mentioned above has to contain an even number of entries and the other contains an odd number; therefore the top row can be viewed as a cyclic permutation of a single palindromic block with the central element of the odd-sized block in the middle. We may consider the cyclic permutation of the matrix which has this middle element in the central column:
\[
\left(
\begin{array}{lcr}
\mbox{[}y_0~ \ldots ~y_{\ell-1}] &  y_{\ell}  & [y_{\ell-1} ~  \ldots~y_{0}] \\
~x_0& \ldots & x_n~
\end{array}
\right).
\]

Since the top row remains unchanged under the operation given by (\ref{eq:matrix}) this matrix is associated to a selfdual $\Gamma$-semimodule if and only if the bottom row remains unchanged as well, that is
\[
\left(
\begin{array}{lcr}
\mbox{[}y_0~ \ldots ~y_{\ell-1}] &  y_{\ell}  & [y_{\ell-1} ~  \ldots~y_{0}] \\
~x_0& \ldots & x_n
\end{array}
\right)
=
\left(
\begin{array}{lcr}
\mbox{[}y_0~ \ldots ~y_{\ell-1}] &  y_{\ell}  & [y_{\ell-1} ~  \ldots~y_{0}] \\
~x_{n-1}& \ldots & x_0~x_n
\end{array}
\right).
\]

Comparison of entries now shows that the sequence $x_0, \ldots , x_{n-1}$ has to be palindromic, therefore the matrix looks like
\begin{equation}\label{eq:matrix1}
\left(
\begin{array}{cccccc}
\mbox{[}y_0~ \ldots ~y_{\ell-1}] &y_{\ell}&[y_{\ell -1}& \ldots &y_1& y_0] \\
\mbox{[}x_0~ \ldots ~x_{\ell-1}]  &[ x_{\ell -1}& x_{\ell -2}&\ldots &x_0]& x_n
\end{array}
\right).
\end{equation}

\noindent Case II: If $n$ is odd the matrix consists of an even number of columns, hence the palindromic blocks mentioned above have to be either both even-sized or both odd-sized:
\medskip

\noindent Case II-a: If both blocks are even-sized the top row can be viewed as a cyclic permutation of a single palindromic block without a central element:
\[
\left(
\begin{array}{lcr}
\mbox{[}y_0~ \ldots ~y_{\ell}]   & [y_{\ell} ~  \ldots~y_{0}] \\
 x_0 ~~ \ldots &\ \ \ \  \ldots  ~ ~x_n
\end{array}
\right).
\]

With the same reasoning as in Case I we deduce that the matrix has to be of the form
\begin{equation} \label{eq:matrix2a}
\left(
\begin{array}{ll}
\mbox{[}y_{0} \ldots \ \ \ldots \ \ \ y_{\ell-1}]&[y_{\ell-1} \ldots ~\ \ \ldots~y_{0}]\\
\mbox{[}x_0~ \ldots ~x_{\ell-2}] ~x_{\ell-1}& [x_{\ell-2} ~  \ldots~x_{0}] \ \ x_n
\end{array}
\right).
\end{equation}

\noindent Case II-b: If both blocks are odd-sized then, up to a cyclic permutation, the top row consists of an odd-sized palindromic block and one additional entry
\[
\left(
\begin{array}{clcr}
y_0&[y_1~ \ldots ~y_{\ell-1}] &  y_{\ell}  & [y_{\ell-1} ~  \ldots~y_{1}] \\
x_0&x_1& \ldots & x_{n}
\end{array}
\right).
\]

Up to a shift by one position to the left, the top row remains unchanged under the operation given by (\ref{eq:matrix}):
\begin{eqnarray*}
&&
\left(
\begin{array}{clcr}
y_0&[y_1~ \ldots ~y_{\ell-1}] &  y_{\ell}  & [y_{\ell-1} ~  \ldots~y_{1}] \\
x_0&x_1& \ldots & x_{n}
\end{array}
\right)^*\\
&~&
\!\!\!\!\!\!=
\left(
\begin{array}{lrrl}
\mbox{[}y_1~ \ldots ~y_{\ell-1}]   \ y_{\ell} \   [y_{\ell-1}  &\ldots&y_{1}] &y_0\\
x_{n-1}& \ldots &  x_{0}&x_n
\end{array}
\right)\\
&&
\!\!\!\!\!\! \equiv
\left(
\begin{array}{clcr}
y_0&[y_1~ \ldots ~y_{\ell-1}] & y_{\ell}& [y_{\ell-1} ~  \ldots~y_{1}] \\
x_n&x_{n-1}& \ldots & x_{0}
\end{array}
\right).
\end{eqnarray*}

Comparison of entries now shows that the bottom row the matrix is a single palindrome:
\begin{equation} \label{eq:matrix2b}
\left(
\begin{array}{cr}
y_0~[y_1~ \ldots ~y_{\ell-1}] &  y_{\ell}  ~ [y_{\ell-1} ~  \ldots~y_{1}] \\
\mbox{[}x_{0} \ldots \ \ \ldots \ \ x_{\ell-1}]&[x_{\ell-1} \ldots ~\ \ \ldots~x_{0}]
\end{array}
\right).
\end{equation}

The structural results above have an immediate consequence for selfdual $\langle \alpha , \beta \rangle$-semimodules in the case of $\alpha$ and $\beta$ odd: 

\begin{proposition}
Let $\Gamma=\langle \alpha , \beta \rangle$ with $\alpha, \beta$ odd. Then any selfdual $\Gamma$-semimodule $\Delta$ is minimally generated by an odd number of elements.
\end{proposition} 

\begin{proof}
The size of the minimal set of generators of $\Delta$ equals the number of columns in the matrix associated to $\Delta$. If this number was even, one of the row sums would be even since by the results above one of the rows has to be palindromic. But the rows sum up to $\alpha$~resp.~$\beta$ and so both row sums are odd by assumption, a contradiction.
\end{proof}

Next we want to count the selfdual semimodules in the case $\alpha, \beta$ odd. By the previous result we only have to count the matrices of the form given by (\ref{eq:matrix1}). Since all entries are positive we have $\sum_{i=0}^{\ell} y_i \leq \lfloor \frac{\alpha}{2} \rfloor$ and $\sum_{i=0}^{\ell} x_i \leq \lfloor \frac{\beta}{2} \rfloor$. Therefore the partial sums $\sum_{j=0}^{r} y_j$ for $r=0, \ldots , \ell$ have to be chosen in the range $1, \ldots , \lfloor\frac{\alpha}{2}\rfloor$. Hence there are ${\lfloor\frac{\alpha}{2}\rfloor \choose \ell}$ different $y$-palindromes of length $\ell$. Similarly there are ${\lfloor\frac{\beta}{2}\rfloor \choose \ell}$ different $x$-palindromes of lenght $\ell$. Any combination of an $x$- and a $y$-palindrome yields a matrix of the form (\ref{eq:matrix1}), and so there are ${\lfloor\frac{\alpha}{2}\rfloor \choose \ell}{\lfloor\frac{\beta}{2}\rfloor \choose \ell}$ of them.

\begin{theorem}\label{5T1}
Let $\Gamma=\langle \alpha , \beta \rangle$ with $\alpha, \beta$ odd. Then there are
\[
{\lfloor\frac{\alpha}{2}\rfloor \choose \ell}{\lfloor\frac{\beta}{2}\rfloor \choose \ell}
\]
selfdual $\Gamma$-semimodules with $2\ell + 1$ generators for $\ell=0,\ldots ,  \lfloor \frac{\alpha}{2}\rfloor$. In total there are
\[
{\lfloor \frac{\alpha}{2}\rfloor + \lfloor \frac{\beta}{2}\rfloor\choose \lfloor \frac{\alpha}{2}\rfloor}
\]
selfdual semimodules.
\end{theorem} 

\begin{proof}
The first assertion was shown above, whereas the second follows by an application of the Vandermonde convolution:
\[
\sum_{\ell=0}^{\lfloor \frac{\alpha}{2}\rfloor} {\lfloor \frac{\alpha}{2}\rfloor \choose \ell}{\lfloor \frac{\beta}{2}\rfloor \choose \ell} =
\sum_{\ell=0}^{\lfloor \frac{\alpha}{2}\rfloor} {\lfloor \frac{\alpha}{2}\rfloor \choose \ell}{\lfloor \frac{\beta}{2}\rfloor \choose \lfloor \frac{\beta}{2}\rfloor -\ell }=
{\lfloor \frac{\alpha}{2}\rfloor + \lfloor \frac{\beta}{2}\rfloor\choose \lfloor \frac{\alpha}{2}\rfloor}.  
\]
\end{proof}

Next we investigate the case of one generator of the semigroup being even. It turns out that this case can be reduced to the case of $\alpha, \beta$ being odd by means of a bijective map. 
\medskip

Let $\alpha$ be even, then $\beta$ is odd since the generators of the semigroup have to be coprime. By our results on the structure of matrices associated to selfdual $\Gamma$-semimodules, in this case the matrix for a selfdual semimodule has to be of the form (\ref{eq:matrix1}) or (\ref{eq:matrix2a}), the form (\ref{eq:matrix2b}) cannot occur since the sum of the entries in the bottom row has to be odd. Next we give a mapping which sends the matrices of selfdual $\langle \alpha,\beta\rangle$-semimodules to those of selfdual $\langle\alpha+1,\beta \rangle$-semimodules. For the matrices of form (\ref{eq:matrix1}) this map is nearly obvious:
\[
\left(
\begin{array}{llllr}
\mbox{[} \ \ldots \ ]  &y_{\ell}&  [&  \ldots    &  ] \\
\mbox{[} \ \ldots \ ]  &[           &  \ldots & ]    & x_n
\end{array}
\right) \mapsto
\left(
\begin{array}{llllr}
\mbox{[} \ \ldots \ ]  &y_{\ell}+1&  [&  \ldots    &  ] \\
\mbox{[} \ \ldots \ ]  &[  \    \ldots    &  \ldots \ & ]    & x_n
\end{array}
\right).
\]

For those matrices of form (\ref{eq:matrix2a}) note that the sum of the bottom row entries being odd, exactly one of the entries $x_{\ell -1}$ and $x_n$ has to be even. Because of the equivalence of matrices under cyclic permutation of columns we may assume that $x_{\ell -1}$ is even, and so the map can be defined by
\[
\left(
\begin{array}{ll}
\mbox{[}\ldots \ \ \ldots \ \ \ ]&[ \ldots ~\ \ \ldots~]\\
\mbox{[} ~ \ldots ~] ~x_{\ell-1}& [ ~  \ldots~ ] \ \ x_n
\end{array}
\right)
\mapsto
\left(
\begin{array}{lcl}
\mbox{[}\ldots \ \ \ldots \ \ \ ]&1&[ \ldots ~\ \ \ldots~]\\
\mbox{[} ~ \ldots ~\ \frac{x_{\ell-1}}{2}]& [\frac{x_{\ell-1}}{2}& ~  \ldots~ ] \ \ x_n
\end{array}
\right).
\]

Hence we have a map $f$ from the set of equivalence classes of matrices associated to selfdual $\langle \alpha,\beta \rangle$-semimodules to those of equivalence classes of matrices associated to selfdual $\langle \alpha+1,\beta \rangle$-semimodules. This map is bijective: One easily checks that the inverse map is given by
\begin{eqnarray*}
&&f^{-1} \left (\left(
\begin{array}{clr}
\mbox{[}~ \ldots ~y_{\ell-1}]  &~y_{\ell}~\ [   y_{\ell-1}~\ldots        &                                    ] \\
\mbox{[}~ \ldots ~x_{\ell-1}]  &[ x_{\ell -1}~ \ldots  \ \             & ]~                                x_n
\end{array}
\right)\right )\\
&=&
\left \{
\begin{array}{ll}
\left(
\begin{array}{clr}
\mbox{[}~ \ldots ~y_{\ell-1}]  &~y_{\ell}-1~\ [   y_{\ell-1}~\ldots        &                                    ] \\
\mbox{[}~ \ldots ~x_{\ell-1}]  &[ x_{\ell -1}~ \ldots  \ \             & ]~                                x_n
\end{array}
\right)&\mbox{for~} y_{\ell}>1 \\
\\
\left(
\begin{array}{ll}
\mbox{[}~ \ldots ~y_{\ell-1}] ~  [   y_{\ell-1}                  &~\ldots  \ \ \ \ \ \ \ \ \ \ \ \  \ ] \\
\mbox{[}~ \ldots ~x_{\ell-2}]~2x_{\ell-1}  & [ x_{\ell -2}~ \ldots  \ \     ]~     x_n
\end{array}
\right)& \mbox{for~} y_{\ell}=1.
\end{array}
\right .
\end{eqnarray*}

\begin{theorem}\label{5T2}
Let $\Gamma=\langle \alpha , \beta \rangle$ with $\alpha$ even and $\beta$ odd. Then there are in total
\[
{\lfloor \frac{\alpha}{2}\rfloor + \lfloor \frac{\beta}{2}\rfloor\choose \lfloor \frac{\alpha}{2}\rfloor}
\]
selfdual semimodules. In particular there are
\[
{\frac{\alpha}{2}-1 \choose \ell}{\lfloor\frac{\beta}{2}\rfloor \choose \ell}
\]
selfdual $\Gamma$-semimodules with $2\ell+1$ generators for $\ell=0,\ldots , \frac{\alpha}{2}-1$, and
\[
{\frac{\alpha}{2}-1 \choose \ell -1}{\lfloor\frac{\beta}{2}\rfloor \choose \ell}
\]
selfdual $\Gamma$-semimodules with $2\ell$ generators for $\ell=1,\ldots , \frac{\alpha}{2}$.
\end{theorem} 

\begin{proof}
The first assertion is clear because of Theorem \ref{5T1} and the bijection constructed above.
The number of selfdual $\Gamma$-semimodules with odd number of generators follows as in the proof of Theorem \ref{5T1}.
Finally, the selfdual $\langle \alpha,\beta \rangle$-semimodules with $2\ell$ generators are in bijection with the matrices of form (\ref{eq:matrix1}) with row sums $\alpha+1$~resp.~$\beta$ and $y_{\ell}=1$. 
This means, $\alpha+1-2\sum_{i=0}^{\ell-1}y_i=1$ or $\sum_{i=0}^{\ell-1}y_i=\frac{\alpha}{2}$.
There are ${\frac{\alpha}{2}-1 \choose \ell -1}$ possibilities for the $y$-block, since the partial sums $\sum_{i=0}^{\ell-2} y_i$ have to be chosen in the range $1, \ldots , \frac{\alpha}{2}-1$, and any of these blocks can be combined with one of the ${\lfloor \frac{\beta}{2} \rfloor \choose \ell}$ $x$-blocks.
\end{proof}

The case of $\alpha$ odd and $\beta$ even can be treated with almost analogous arguments. Now 
the matrix for a selfdual semimodule has to be of the form (\ref{eq:matrix1}) or (\ref{eq:matrix2b}), and the mapping which sends the matrices of selfdual $\langle \alpha,\beta\rangle$-semimodules to those of selfdual $\langle\alpha,\beta +1 \rangle$-semimodules can be defined by
\[
\left(
\begin{array}{llllr}
\mbox{[} \ \ldots \ ]  &y_{\ell}&  [&  \ldots    &  ] \\
\mbox{[} \ \ldots \ ]  &[           &  \ldots & ]    & x_n
\end{array}
\right) \mapsto
\left(
\begin{array}{llllr}
\mbox{[} \ \ldots \ ]  &y_{\ell}&  [&  \ldots    &  ] \\
\mbox{[} \ \ldots \ ]  &[  \    \ldots    &  \ldots \ & ]    & x_n+1
\end{array}
\right).
\]
for the first form and by
\[
\left(
\begin{array}{cr}
y_0~[y_1~ \ldots ~y_{\ell-1}] &  y_{\ell}  ~ [y_{\ell-1} ~  \ldots~y_{1}] \\
\mbox{[}x_{0} \ldots \ \ \ldots \ \ x_{\ell-1}]&[x_{\ell-1} \ldots ~\ \ \ldots~x_{0}]
\end{array}
\right)
\mapsto
\left(
\begin{array}{lcl}
\frac{y_{0}}{2} ~ [\ldots \ \ \ldots \ \ \ ]&y_{\ell}&[ \ldots ~\ \ \ldots~]~\frac{y_{0}}{2}\\
\mbox{[} ~ \ldots ~\ \ \ \ \  \ \ldots]& [\ldots& ~  \ldots~ \ \ \ \ \ \ \ ] \ \  1
\end{array}
\right)
\]
for the second form. Similar to the previous case we obtain:

\begin{theorem}\label{5T3}
Let $\Gamma=\langle \alpha , \beta \rangle$ with $\alpha$ odd and $\beta$ even. Then there are in total
\[
{\lfloor \frac{\alpha}{2}\rfloor + \lfloor \frac{\beta}{2}\rfloor\choose \lfloor \frac{\alpha}{2}\rfloor}
\]
selfdual semimodules. In particular there are
\[
{\lfloor\frac{\alpha}{2}\rfloor \choose \ell}{\frac{\beta}{2}-1 \choose \ell}
\]
selfdual $\Gamma$-semimodules with $2\ell+1$ generators for $\ell=0,\ldots ,\lfloor \frac{\alpha}{2}\rfloor$, and
\[
{\lfloor\frac{\alpha}{2}\rfloor \choose \ell}{\frac{\beta}{2}-1 \choose \ell -1}
\]
selfdual $\Gamma$-semimodules with $2\ell$ generators for $\ell=1,\ldots , \lfloor\frac{\alpha}{2} \rfloor$.
\end{theorem} 

The previous theorems in particular recover a result already given by Gorsky and Mazin \cite[Theorem~2.6]{gm}:

\begin{corollary}
There are
$
{\lfloor \frac{\alpha}{2}\rfloor + \lfloor \frac{\beta}{2}\rfloor\choose \lfloor \frac{\alpha}{2}\rfloor} 
$
selfdual $\langle \alpha , \beta \rangle$-semimodules.
\end{corollary}

\section{Relationship between $\Delta^{*}$ and $\mathrm{Syz}(\Delta)$}\label{section6}

We conclude this paper with some remarks concerning the connection between the dual of the $\langle \alpha, \beta \rangle$-semimodule $\Delta$ and the semimodule $\mathrm{Syz}(\Delta)$ introduced in our previous paper \cite{mu2}. We recall the main properties of $\mathrm{Syz}$.
If $\Delta$ is generated by an $\langle \alpha, \beta \rangle$-lean set $I$, then
\[
\mathrm{Syz}(\Delta):=\bigcup_{\substack{i,i' \in I\\i \neq i'}} \left ( \big(\langle \alpha, \beta \rangle+ i \big) \cap \big(\langle \alpha, \beta \rangle + i' \big)  \right ).
\]
The semimodule $\mathrm{Syz}(\Delta)$ consists of those elements in $\Delta$ which admit more than one presentation of the form $i + x$ with $i \in I, x \in \langle \alpha, \beta \rangle$, and it has also a meaning in terms of lattice paths: as mentioned in the previous section the elements of $I$ correspond to the ES-turning points of a lattice path from $(0,\beta)$ to $(\alpha, 0)$ and the SE-turning points of this path can be identified with the elements minimally generating $\mathrm{Syz}(\Delta)$.

\begin{lemma}
The minimal sets of generators of $\Delta_I^{*}$ and $\mathrm{Syz}(\Delta_I)$ are in correspondence via the map $x \mapsto \alpha \beta -x$.
\end{lemma}

\begin{proof}
Let $J$ be the minimal system of generators of $\mathrm{Syz}(\Delta_I)$, then, by Theorem 4.2 of \cite{mu2}, $[I,J]$ is a fundamental couple in the sense of \cite[Definition 3.10]{mu}.  Let $I=\{i_0=0, i_1, \ldots , i_n\}$ with $i_k=\alpha \beta -a_k \alpha -b_k \beta$. According to \cite[Corollary 3.21(c)]{mu} the elements of $J$ are given by 
\begin{eqnarray*}
j_0&=& (\beta -a_1)\alpha\\
j_k&=& \alpha \beta - a_{k+1}\alpha - b_k \beta \ \ \ \mbox{~for~} k =1, \ldots, n-1\\
j_n&=&(\alpha - b_n)\beta,
\end{eqnarray*}
so Theorem \ref{3T1} yields the conclusion.
\end{proof}

As with the dual, the matrix description of $\mathrm{Syz}(\Delta)$ can be obtained from those of $\Delta$ by certain permutation of entries: 
\[
\mathrm{Syz}\left (\left(
\begin{array}{cccc}
y_0 & y_1 & \ldots & y_n   \\
x_0 & x_1 & \ldots & x_n 
\end{array}
\right) \right )=\left(
\begin{array}{ccccc}
y_1 & y_2 & \ldots & y_n& y_0   \\
x_ 0& x_1 & \ldots & x_{n-1} & x_n 
\end{array}
\right).
\]
(Note that $\mathrm{Syz}(\Delta)=\Delta^*$ for $n\leq1$.) Therefore it is easy to compute the effect of a composition of the operators $\mathrm{Syz}(-)$ and $(-)^*$ in terms of the matrix description, revealing an interesting relation:
\[
\begin{array}{lcc}
\mathrm{Syz}\left (\left(
\begin{array}{cccc}
y_0 & y_1 & \ldots & y_n   \\
x_0 & x_1 & \ldots & x_n 
\end{array}
\right)^* \right )&\equiv& 
\mathrm{Syz}\left ( \left(
\begin{array}{cccc}
y_n &  \ldots & y_1 &  y_0   \\
x_ {n-1}&  \ldots & x_{0} & x_n 
\end{array}
\right) \right )\\
&&\\
\equiv \left(
\begin{array}{cccc}
y_{n-1} & \ldots & y_0 & y_n   \\
x_{n-1} & \ldots & x_0 & x_n 
\end{array}
\right)&\equiv&
\left(
\begin{array}{cccc}
y_{n} & y_0& \ldots & y_{n-1}   \\
x_{0} & x_1 & \ldots & x_n 
\end{array}
\right)^*\\
&&\\
\equiv \left ( \mathrm{Syz}^{-1}\left(
\begin{array}{cccc}
y_0 & y_1 & \ldots & y_n   \\
x_0 & x_1 & \ldots & x_n 
\end{array}
\right)\right )^*,&&
\end{array}
\]
in other words:

\begin{proposition}
The operators $\mathrm{Syz}(-)$ and $(-)^*$ generate a dihedral group of order  $2n$ acting on 
the set of isomorphism classes of $\langle \alpha, \beta \rangle$-semimodules with $n$ generators, where $n \geq 3$.
\end{proposition}

\end{document}